\documentclass{article}
\usepackage{indentfirst}
\usepackage{color}
\usepackage{amsthm}
\usepackage{amsmath}
\usepackage{amssymb}
 \usepackage{amsfonts}
\usepackage{latexsym}
 \def\qed{\hfill $\Box$}
\title{Jump processes on the boundaries of random trees}
\author{Yuki Tokushige}

\begin{document}
\newtheorem{Definition}{Definition}[section]
\newtheorem{Proposition}[Definition]{Proposition}
\newtheorem{DefAndProp}[Definition]{DefAndProp}
\newtheorem{Theorem}[Definition]{Theorem}
\newtheorem{Assumption}[Definition]{Assumption}
\newtheorem{Lemma}[Definition]{Lemma}
\newtheorem{Remark}[Definition]{Remark}
\newtheorem{Example}[Definition]{Example}
\newtheorem{Corollary}[Definition]{Corollary}
\newtheorem{Problem}[Definition]{Problem}
\newtheorem{Recall}[Definition]{Recall}
\newtheorem{myremark}[Definition]{myremark}
\newtheorem{Notation}[Definition]{Notation}
\newtheorem{Proof}[Definition]{Proof}
\newtheorem{Application}[Definition]{Application}

\newtheorem*{notat}{Notation}
\renewcommand{\thefootnote}{\fnsymbol{footnote}}

\makeatletter
    \renewcommand{\theequation}{%
    \thesection.\arabic{equation}}
    \@addtoreset{equation}{section}
  \makeatother

\maketitle
\begin{abstract}
In \cite{K}, Kigami showed that a transient random walk on a deterministic infinite tree $T$
 induces its trace process on the Martin boundary of $T$. 
 In this paper, we will deal with trace processes on Martin boundaries of random trees instead of deterministic ones,
  and prove short time log-asymptotic of heat kernel estimates and estimates of mean displacements.
\end{abstract}

\section{Introduction}
\footnote[0]{2010 {\it Mathematics Subject Classification}.  60J25, 60J50.}
\footnote[0]{{\it Key words and phrases}. Galton-Watson tree, Dirichlet form, Martin boundary.}
 Consider an infinite tree $T$ and a transient random walk $\{Z_n\}_{n\geq0}$ on $T$.
 The transient random walk on $T$ finally hits its Martin boundary, which is the collection of ``infinities''.
 It is well-known that under suitable assumptions, the transient random walk on $T$ induces a Hunt process
 (equivalently a {\it Dirichlet form}) on its Martin boundary $M$ in the following way:
 let $(\mathcal{E},\mathcal{F})$ be the Dirichlet form associated with $\{Z_n\}_{n\geq0}$
 and ${\rm HARM}_T$ be its first hitting distribution
 (called the {\it harmonic measure}) to $M$ started from a certain point in $T$.
 By the theory of Martin boundaries, we have the map $H$ which transforms functions on $M$ into functions
 on $T$ in such a way that for a given function $f$ on $M$, $Hf$ is harmonic on $T$ and has the boundary value $f$ on $M$.
 Then the induced form $(\mathcal{E}_M,\mathcal{F}_M)$ on the Martin boundary $M$ is given by 
\begin{align*}
 \mathcal{F}_M&:=\{f\in L^2(M,{\rm HARM}_T): Hf\in\mathcal{F}\},\\
 \mathcal{E}_M(f,g)&:=\mathcal{E}(Hf,Hg)\ {\rm for}\ f,g\in\mathcal{F}_M.
  \end{align*} 
Since $Hf$ solves the Dirichlet problem at ``infinity'', $(\mathcal{E}_M,\mathcal{F}_M)$ can be regarded as
 the trace of $(\mathcal{E},\mathcal{F})$ on $M$. In \cite{K}, Kigami constructed a Hunt process
 $\{X_t\}_{t>0}$ on $M$ associated with $(\mathcal{E}_M,\mathcal{F}_M)$ and obtained estimates
 of its heat kernels $p_t(\cdot,\cdot)$ for a deterministic tree.
 In particular, detailed two sided heat kernel estimates are obtained when ${\rm HARM}_T$ has the volume doubling property
 with respect to the {\it intrinsic metric} $D$ on $M$, which will be defined in Definition \ref{intrinsic}.
 We refer to \cite{B} for the history and related topics.\par
 We now give a classical example to which the above construction of jump processes is analogous.
 Consider the reflected Brownian motion on the unit disc
 $\mathbb{D}:=\{(x,y)\in\mathbb{R}^2\ ;\ x^2+y^2<1\}$. Note that the corresponding Dirichlet form
  $(E,F)$ is given by
  \begin{align*}
  E(u,v)&:=
  \int_{\mathbb{D}}\left(\frac{\partial u}{\partial x}\frac{\partial v}{\partial x} +\frac{\partial u}{\partial y}
  \frac{\partial v}{\partial y}\right)dxdy,\\
  F&:=\left\{u\in L^2(\mathbb{D})\ ;\  \frac{\partial u}{\partial x},\frac{\partial u}{\partial y}\in L^2(\mathbb{D})\right\}.
  \end{align*}
  Let $H$ be an operator of taking the {\it Poisson integral} of a given function 
  $\varphi:\partial\mathbb{D}\rightarrow\mathbb{R}$, which
  is defined as follows:
  \begin{align*}
  H\varphi(re^{i\theta}):=\int_0^{2\pi}\frac{1-r^2}{1-2r\cos(\theta-\theta')+r^2}\varphi(\theta')d\nu(\theta'),
  \end{align*}
  where $\nu$ is the normalized uniform measure on $\partial\mathbb{D}$. Note that the probability measure $\nu$
 coincides with the hitting distribution on $\partial\mathbb{D}$ of the Brownian motion starting at $0$ due to its rotation invariance. Now we define a quadratic form $(E_{\partial\mathbb{D}},F_{\partial\mathbb{D}})$ by
 \begin{align*}
 E_{\partial\mathbb{D}}(\varphi,\psi)&:=E(H\varphi,H\psi),\\
 F_{\partial\mathbb{D}}&:=\{\varphi\in L^2(\partial\mathbb{D},\nu)\ ;\ H\varphi\in F\}.
 \end{align*}
 It is well-known that $(E_{\partial\mathbb{D}},F_{\partial\mathbb{D}})$ yields a regular Dirichlet form
 on $L^2(\partial\mathbb{D},\nu)$, and it corresponds to the trace process of the reflecting Brownian motion
 on $\partial\mathbb{D}$. We remark that $E_{\partial \mathbb{D}}$ has the following explicit expression known as
 the {\it Douglas integral}:
 \begin{align*}
 E_{\partial \mathbb{D}}(\varphi,\psi)=\frac{\pi}{4}\int_0^{2\pi}\int_0^{2\pi}
 \frac{(\varphi(\theta)-\varphi(\theta'))(\psi(\theta)-\psi(\theta'))}{\sin^2(\frac{\theta-\theta'}{2})}d\nu(\theta)d\nu(\theta').
 \end{align*}
 In the context of potential theory on Euclidean domains, 
 the general analogue of the Douglas integral was obtained 
 in \cite{D},
 where the kernel $1/\sin^2(\frac{\theta-\theta'}{2})$ of $E_{\partial\mathbb{D}}$ is replaced by the
 {\it Naim kernel} which was introduced in \cite{N}.
 Later, in the setting of the Martin boundary of reversible Markov chains on discrete
 graphs, Silverstein (\cite{S}) studied a similar problem. See the references introduced above for details.
 \par
In this paper, we will consider random trees instead of deterministic ones,   
 and we are going to study properties of processes on the Martin boundary
 induced by transient random walks on random trees. 
  In particular, we are interested in random trees generated by branching processes.
 In \cite{K}, it is assumed to analyze properties of processes on the boundary that
 the harmonic measure satisfies the volume doubling property with respect to the intrinsic metric $D$.
 But for random trees, the volume doubling property of the harmonic measure does not hold in general. 
 We overcome this difficulty by utilizing the ergodic theory
 on the space of trees developed in \cite{LPP1} and \cite{LPP2}. \par 
 We now explain the framework more precisely.
 Consider a Galton-Watson branching process with offspring distribution $\{p_k\}_{k\geq0}$.
 Starting from a single individual called the {\it root}, which is denoted by $o$,
 this process yields a random tree $\mathcal{T}$,
 which is called a Galton-Watson tree with offspring distribution $\{p_k\}_{k\geq0}$.
 In this paper, we assume that $p_0=0$ and $\mathcal{T}$
 is supercritical (namely $m:=\sum_{k\geq0}kp_k>1$) to guarantee that $\mathcal{T}$ is almost surely infinite. 
 Under the above assumptions, the Galton-Watson
 tree $\mathcal{T}$ can be regarded as the $\mathbb{T}$-valued random variable,
 where $\mathbb{T}:=\{T\ ;\ T\ {\it\ is\ an\ infinite\ rooted\ tree}\}$, and we will denote the distribution of $\mathcal{T}$
 by $\mathbb{P}_{{\rm GW}}$.
 The structure of $\mathcal{T}$ and its electric network has been studied extensively for many years:
 see \cite{LP} for references and details.
 Given a rooted tree $T$, we consider $a\ \lambda$-{\it biased random walk} on $T$ under the probability measure
 $P_{\lambda}^{T}$. Precisely speaking, for $\lambda>0$,
 we define a Markov chain $\{Z_n^{\lambda}\}_{n\geq0}$
 on the vertices of $T$ such that if $u\neq o$, $u$ has $k$ children $u_1,...,u_k$ and the parent $\pi(u)$, then
\begin{eqnarray*}
P_{\lambda}^{T}(Z_{n+1}^{\lambda}=\pi(u)\ |\ Z_n^{\lambda}=u)&=&\frac{\lambda}{\lambda+k},\\
P_{\lambda}^{T}(Z_{n+1}^{\lambda}=u_i\ |\ Z_n^{\lambda}=u)&=&\frac{1}{\lambda+k},\ \ \ {\rm for}\ 1\le i\le k,
\end{eqnarray*}
and if $u=o$, the random walk moves to its children equally likely. 
 It is proved in \cite{L} that $\{Z_n^{\lambda}\}_{n\geq0}$ on the supercritical Galton-Watson tree $\mathcal{T}$
 is transient for almost every $\mathcal{T}$,
 if and only if $0<\lambda<m$. Thus for $0<\lambda<m$, we have the harmonic
 measure ${\rm HARM}_{\mathcal{T}}^{\lambda}$, the induced Dirichlet form $(\mathcal{E}^{\lambda},\mathcal{F}^{\lambda})$
 $\mathbb{P}_{{\rm GW}}$-a.s.
 Moreover,
 we get the heat kernels $p_t^{\lambda}(\cdot,\cdot)$ and
 the Hunt process $\{X_t^{\lambda}\}$ on the boundary of $\mathcal{T}$ $\mathbb{P}_{{\rm GW}}$-almost surely,
 which are associated with $(\mathcal{E}^{\lambda},\mathcal{F}^{\lambda})$.
  In \cite{LPP1} and \cite{LPP2}, Lyons, Pemantle and Peres showed that for $0<\lambda<m$,
 $\beta_{\lambda}:=\dim{\rm HARM}_{\mathcal{T}}^{\lambda}$ is a deterministic constant for almost every $\mathcal{T}$,
 see Theorem \ref{HM}.\par
 We now state the results on short time log-asymptotic of heat kernel estimates and estimates
 of mean displacements.
 Note that $d(\cdot,\cdot)$ is the natural metric on $M$ defined in Definition \ref{def} and that
 we will prove in Corollary \ref{>} that $\beta_{\lambda}-\log\lambda>0$ for $0<\lambda<m$.
\begin{Theorem}\label{OND}For $0<\lambda<m$, the following results hold $\mathbb{P}_{{\rm GW}}\mathchar`-a.s.$
\begin{align}
\lim_{t \to 0}\frac{\log p_t^{\lambda}(\omega,\eta)}{\log t}&=1,\ \ \ \ {\rm for\ any}\ \omega,\eta\in M\ 
{\rm with}\ \omega\neq\eta,\label{A}\\
-\lim_{t \to 0}\frac{\log p_t^{\lambda}(\omega,\omega)}{\log t}&= \frac{\beta_\lambda}{\beta_\lambda - \log\lambda},\ \ \ \  {\rm HARM}_{\mathcal{T}}^{\lambda}\ a.e.\mathchar`-\omega.\label{AA}
\end{align}
\end{Theorem}

\begin{Theorem}\label{displacement}For $0<\lambda<m$ and $\gamma>0$,
 the following holds $\mathbb{P}_{{\rm GW}}\mathchar`-a.s.$
\begin{eqnarray*}
\lim_{t \to 0} \frac{\log{E^{\lambda}_{\omega}[d(\omega,X_t^{\lambda})^{\gamma}]}}{\log t}
=\left(\frac{\gamma}{\beta_{\lambda}-
\log\lambda }\right)\wedge1,\ \ \ \ {\rm HARM}_{\mathcal{T}}^{\lambda}\ a.e.\mathchar`-\omega,
\end{eqnarray*}
where $E^{\lambda}_{\omega}$ denotes the (quenched) expectation with respect to the probability distribution
 of $\{X_t^{\lambda}\}$ starting at $\omega$.
\end{Theorem}
 
 Since the volume doubling property of the harmonic measure, which is assumed to analyze
 the heat kernels in \cite{K}, holds only when $p_1=0$ and $\sup\{n:p_n>0\}<\infty$, the heat kernel estimates proved in
 \cite{K} cannot be applied for this problem in general.
 The proofs of Theorem \ref{OND} and Theorem \ref{displacement} both utilize the explicit expression (\ref{expression})
 for the heat kernels
 obtained in \cite{K} which involves the harmonic measure and the effective resistance.
 Therefore, we will obtain an estimate of the effective resistance 
 for a Galton-Watson tree (Proposition \ref{Resistance})
 by using the ergodic theory on the space of trees, and apply it for the expression of the heat kernels 
 together with an estimate of the harmonic measure
 of a Galton-Watson tree (Theorem \ref{HM}) obtained in \cite{LPP1}, \cite{LPP2}.
 In order to prove Theorem \ref{displacement}, we will establish an analogous result on mean displacement with respect to
 $D$ (Proposition \ref{Dis}), and use a comparison between two metrics $d$ and $D$.\par
 Note that Theorem \ref{OND} and Theorem \ref{displacement} imply that the spectral dimension (resp. the walk dimension) is
 $2\beta_\lambda/(\beta_\lambda - \log\lambda)$ (resp. $(\beta_{\lambda}-\log\lambda)\vee1$).
 We remark here that long time asymptotics of heat kernels $p_t^{\lambda}(\cdot,\cdot)$ are trivial
 because of compactness of the boundary.\par
 This paper is organized as follows. In Section 2, we will introduce notation and results in \cite{K}.
 In Section 3, we will introduce notation and results on Galton-Watson trees
 studied in \cite{LPP1}, \cite{LPP2} and \cite{Lin}, and prove the asymptotic of the effective resistance along infinite rays,
 which will be important for the proof of the main results.
 We then prove the lower bound for the dimension of the harmonic measures,
 which is of independent interest.
 In Section 4, we will give the proofs of our main results.\\
 
{\bf Acknowledgments.} The author would like to thank Professor Takashi Kumagai for detailed discussions
 and careful readings of earlier versions of this paper, Professor Ryoki Fukushima 
 for the literature information about the random walks in random environment.
 Special thanks go to Shen Lin for informing the author that his results in \cite{Lin} can simplify the argument
 in Section 3 of the first version of this paper. This research is partially supported by JSPS KAKENHI 16J02351.
 
\section{Preliminaries and Kigami's results}
 In this section, we will introduce some notation and the results studied in \cite{K}.
\subsection{Weighted graphs and associated random walks}
\begin{Definition}
(1) A pair $(V,C)$ is called a weighted graph (an electric network) if $V$ is a countable set and
 $C:V\times V\rightarrow[0,\infty)$ satisfies $C(x,y)=C(y,x)$ for any $x,y\in V$ and $C(x,x)=0$ for any $x\in V.$
 In what follows, we always assume that $C(x):=\sum_{y \in V}C(x,y)>0$ for any $x\in V$.
 The points of $V$ are called
 the vertices of the graph $(V,C).$ Two vertices $x,y\in V$ are said to be adjacent if and only if $C(x,y)>0.$\\
(2) A weighted graph $(V,C)$ is called connected if and only if for any $x,y \in V,$ there exists a sequence  of vertices of $V\ x=x_0,x_1,x_2,\cdot \cdot \cdot, x_n=y$ such that $x_k$ and $x_{k+1}$ are adjacent for $0\le k \le n-1.$\\
(3) A weighted graph $(V,C)$ is called locally finite if and only if $\#\{ y \in V: {\it y\ is\ adjacent\ to\ x}\}<+ \infty$
 for all $x \in V$.
\end{Definition}

In this paper, we always assume that the weighted graph $(V,C)$ is connected and locally finite.
 A weighted graph defines a reversible Markov chain on $V$ in the following way.

\begin{Definition} Define $p(x,y):=\frac{C(x,y)}{C(x)}.$ For $n\geq 0,$
 we define $p^{(n)}(x,y)$ for $x,y \in V$ inductively by $p^{(0)}(x,y):={\bf 1}_x(y)$ and 
\begin{equation*}
p^{(n+1)}(x,y):=\sum_{z \in V}p^{(n)}(x,z)p(z,y).
\end{equation*}
Define $G(x,y):=\sum_{n=0}^{\infty}p^{(n)}(x,y)\in[0,\infty]$. $G(x,y)$ is called the Green function of $(V,C).$ A weighted graph $(V,C)$ is said to be transient if  and only if $G(x,y)<+\infty$ for any $x,y \in V.$
\end{Definition}

Let $(\{ Z_n\}_{n\geq 0},\{P_x\}_{x \in V})$ be the random walk on $V$ associated with $(V,C),$
 \ that is $P_x(Z_n=y)=p^{(n)}(x,y)$.

\begin{Definition}
(1) Define $l(V)$ to be the set of all $\mathbb{R}$-valued functions on $V$.
The Laplacian $\Delta:l(V)\rightarrow l(V)$ associated with $(V,C)$ is defined by
\begin{align*}
\Delta u(x):=\sum_{y \in V}p(x,y)(u(y)-u(x)),
\end{align*}
 for any $u \in l(V)$.
 A function $u \in l(V)$ is said to be harmonic on $V$ with respect to $(V,C)$ if and only if $\Delta u(x)=0$ for any $x\in V. $ Define
\begin{align*}
\mathcal{H}(V,C)&:=\{u \in l(V):{\it u\ is\ harmonic\ on\ V}\},\\
\mathcal{H}^{\infty}(V,C)&:=\{u \in \mathcal{H}(V,C):{\it u\ is\ bounded}\}.
\end{align*}
(2) Define $\mathcal{F}:=\{u \in l(V):\sum_{x,y \in V}C(x,y)(u(x)-u(y))^2<+\infty \}.$ For any $u,v \in \mathcal{F},$ define\
\begin{equation}\mathcal{E}(u,v):=\frac{1}{2}\sum_{x,y \in V}C(x,y)(u(x)-u(y))(v(x)-v(y)).\nonumber\end{equation}
The bilinear form $(\mathcal{E},\mathcal{F})$ is called the resistance form associated to $(V,C)$.
\end{Definition}

In the rest of the section, we introduce the notion and  fundamental results of Martin boundaries
 of transient weighted graphs. See \cite[Chapter 7]{W} for references and details.

\begin{Definition}
 Assume $(V,C)$ is transient. The Martin kernel of $(V,C)$ is $K_z(x,y):=G(x,y)/G(z,y)$ for $x,y,z \in V.$ 
\end{Definition}

\begin{Proposition}\cite[Theorem 9.18.]{W}\ Assume $(V,C)$ is transient. Then there exists a unique minimal compactification $\tilde{V}$ of $V\ ({\it up\ to\ homeomorphism})$ such that $K_z$ extends
 to a continuous function from $V\times \tilde{V}\ to\ \mathbb{R}.$ 
 $\tilde{V}$ is independent of the choice of $z.$
 $\tilde{V}$ is called the Martin compactification of $V.$
 Moreover, there exists a $\tilde{V}\backslash V$-valued\ random\ variable\ $Z_{\infty}$
 such that $P_x(\lim_{n \to \infty}Z_n=Z_{\infty})=1$ for\ any\ $x\in V$.
\end{Proposition}

\begin{Definition} Assume $(V,C)$ is transient.\ 
 Define $M(V,C):=\tilde{V}\backslash V$, which is called the Martin boundary of $(V,C).$ Define a probability measure 
${\rm HARM}_{V,x}$ on $M(V,C)$ by
 ${\rm HARM}_{V,x}(B):=P_x(Z_{\infty} \in B)$
for any Borel set $B\subseteq M(V,C)$.\ The probability measure 
${\rm HARM}_{V,x}$ on the Martin boundary $M(V,C)$ is called the harmonic measure of $(Z_n)_{n\geq 0}$ starting from $x$.
 The harmonic measure actually depends on the weight $C$, but we will denote it by ${\rm HARM}_{V,x}$
 for simplicity of notation when the choice of the weight is clear form the context.
\end{Definition}

The following theorem gives the representation of harmonic functions on $(V,C).$
\begin{Theorem}\cite[Theorem 9.37.]{W} Assume that $(V,C)$ is transient.\\
(1) $K_z(\cdot,y) \in \mathcal{H}(V,C)$ for any $z\in V$ and any $y \in \tilde{V}$.\\
(2) If $g \in \mathcal{H}^{\infty}(V,C)$, then there exists $f \in L^{\infty}(M(V,C),{\rm HARM}_{V,z})$ such that
\begin{equation*}
g(x)=\int_{M(V,C)}K_z(x,y)f(y)d{\rm HARM}_{V,z}(y).
\end{equation*}
 Note that the function $f$ does not depend on the choice of $z$ since by connectedness of $(V,C)$, the harmonic measures
 ${\rm HARM}_{V,x}$ and ${\rm HARM}_{V,z}$
 are mutually absolutely continuous for $x,z\in V$, and 
 \begin{align*}
 K_z(x,\cdot)=\frac{d{\rm HARM}_{V,x}}{d{\rm HARM}_{V,z}}(\cdot). 
 \end{align*} 
\end{Theorem}
\subsection{Transient trees and their Martin boundaries}
We now consider transient trees and their Martin boundaries. 
\begin{Definition}
A weighted graph $(T,C)$ is called a tree if and only if it is connected and does not have cycles.
 If $(T,C)$ is a tree, then
 for all $x,y\in T,$\ there exists a unique path between $x$ and $y$ with the minimal number of edges,
 which is independent of the weight $C$.
 We will denote by $\overline{xy}$ the shortest path between $x$ and $y$.
\end{Definition}
$(T,C)$ is said to be {\it rooted} when it has a fixed reference point, which will be denoted by $o.$
 In the rest of this paper, we always assume that $(T,C)$ is a rooted tree.

\begin{Definition}\label{def}
(1) An infinite path $(x_0,x_1,\cdot \cdot \cdot)\in T^{\mathbb{N}}$ is said to be an infinite ray from $x\in T$ if and only if $x_0=x$\ and$\ (x_0,x_1,\cdot \cdot \cdot,x_n)$ is the shortest path between from $x$ and $x_n$ for all $n\geq1$.\\
(2) For $x\in T,$\ define\ the\ height\ of\ $x$,\ $h(x)$\ by\ the\ length\ of\ the\ shortest\ path\ between\ $o$\ and\ $x$. Define $T_k:=\{x\in T;\ h(x)=k\}$ for $k\in\mathbb{N}$.\\
(3) For $x \in T,$ Define $N(x):=\{y \in T:y\ $is\ adjacent\ to$\ x \}.$ For $x\neq o$,
the parent of $x$, which will be denoted by $\pi(x)$, is the unique element of $N(x)$ which satisfies $h(\pi(x))=h(x)-1$.
 We set $S(x):=N(x)\backslash\{\pi(x)\}.$\\
(4) Define $\Sigma:=\Sigma(T,C)$\ to be\ the\ collection\ of\ infinite\ rays\ from\ $o$
 and $\hat{T}:=T\cup \Sigma(T,C).$ \\
 (5) For $x,y\in T$, define $N(x,y)\in\mathbb{N}$ by
 \begin{align*}
 N(x,y):=\frac{h(x)+h(y)-|x,y|}{2},
 \end{align*}
 where $|x,y|$ denotes the length of the shortest path between $x$ and $y$.
 We can extend $N(\cdot,\cdot)$ to $\hat{T}\times\hat{T}$ in the following manner:
 For $\omega=(\omega_0,\omega_1,\cdot\cdot\cdot)\in \Sigma,$
 define $[\omega]_n:=\omega_n$ for $n\geq 0$.
 For $\omega,\eta\in\Sigma$, we define
 \begin{align*}
 N(\omega,\eta):=\lim_{n\to\infty}\frac{h([\omega]_n)+h([\eta]_n)-|[\omega]_n,[\eta]_n|}{2}.
  \end{align*}
  It is easy to see that the limit exists.
  For $x\in T$ and $\omega\in\Sigma$, we can define $N(x,\omega)$ similarly.
 Notice that for $\omega,\eta\in\Sigma$, $N(\omega,\eta)$ can be expressed as follows:
   \begin{align*}
   N(\omega,\eta)=\max\{n \geq 0:[\omega]_n=[\eta]_n\}.
   \end{align*}
 We now define $[\omega,\eta]:=[\omega]_{N(\omega,\eta)}=[\eta]_{N(\omega,\eta)}$ for $\omega,\eta \in \Sigma$.\\
 (6) For $z,w\in\hat{T}$,
 let $d(z,w):=e^{-N(z,w)}$ with the convention $e^{-\infty}=0$.
 Then $d(\cdot,\cdot)$ defines an ultrametric on $\hat{T}.$\ Define $B_d(\omega,r):=\{\eta \in \Sigma:d(\omega,\eta)<r\}.$
\end{Definition}
The following theorem due to \cite{C} is a fundamental result on the Martin boundary of a tree. 

\begin{Theorem}\cite[Theorem 9.22.]{W} Assume $(T,C)$ is transient.
 Then the Martin compactification\ $\overline{T}$\ of\ $T$ is  always homeomorphic to $(\hat{T},d).$
\end{Theorem}

By the above theorem, we will identify the Martin boundary $M(T,C)$ with $\Sigma$, then $(\Sigma,d)$ is compact.
 In the rest of this article, we will always assume the following condition.
\begin{Assumption}\label{assumption}
$(T(x),C|_{T(x)})$ is transient for any $x\in T$ where $T(x)=\{y\in T:x \in \overline{o y}\}$ and
$C|_{T(x)}$ is the restriction of $C$ to $T(x)$.
\end{Assumption}

\subsection{The jump process on the boundary of a deterministic tree}

In what follows, we will write $K(\cdot,\cdot)=K_o(\cdot,\cdot)$ and ${\rm HARM}_{T}={\rm HARM}_{T,o}$
when $(T,C)$ is a rooted tree.

\begin{Definition} Define a linear map $H:L^1(\Sigma,{\rm HARM}_{T})\rightarrow l(T)$ by
\begin{equation}Hf(x):=\int_{\Sigma}K(x,y)f(y)d{\rm HARM}_{T}(y) \nonumber \end{equation}
for any $x\in T$ and $f\in L^1(\Sigma,{\rm HARM}_{T}).$ Moreover,
 define $\mathcal{F}_{\Sigma}:=\{f \in L^1(\Sigma,{\rm HARM}_{T}):Hf\in \mathcal{F} \}$
 and $\mathcal{E}_{\Sigma}(f,g):=\mathcal{E}(Hf,Hg)$ for any $f,g\in \mathcal{F}_{\Sigma}.$
\end{Definition}

In \cite{K}, Kigami studies various properties of the quadratic form $(\mathcal{E}_{\Sigma},\mathcal{F}_{\Sigma}).$
 In particular, the following result is established. 
\begin{Theorem}\label{DF}\cite[Theorem 5.6.]{K} $(\mathcal{E}_{\Sigma},\mathcal{F}_{\Sigma})$
 is a regular Dirichlet form on $L^2(\Sigma,{\rm HARM}_{T}).$
\end{Theorem}
By the above theorem, there exists a stochastic process on $\Sigma$
 which corresponds to $(\mathcal{E}_{\Sigma},\mathcal{F}_{\Sigma})$.
 Before explaining the results on the properties of this process studied in \cite{K},
 we introduce the intrinsic metric on the boundary $\Sigma$, and conditions which tell us when the harmonic measure
 satisfies the volume doubling property with respect to the metric.
\begin{Definition}\label{intrinsic}
 Define $D(x)={\rm HARM}_{T}(\Sigma(x))R(x)$ for $x\in T$, where
 $\Sigma(x)=\{\omega\in \Sigma:\ [\omega]_n=x\ {\it for\ some}\ n\geq 0\}$
 and $R(x)$ is the effective resistance from $x$ to $\Sigma(x)$ in $T(x)$.
 Note that Assumption \ref{assumption} guarantees that $R(x)<+\infty$\ and\ $D(x)<+\infty$\ for\ any\ $x\in T$.
 For $\omega \neq \eta \in \Sigma$, define $D(\omega,\eta)=D([\omega,\eta])$ and $D(\omega,\omega)=0$ for any $\omega \in \Sigma.$
\end{Definition}

\begin{Proposition}\label{D}\cite[Proposition 6.4.]{K}\\
(1) For any $\omega\in\Sigma$, $\{D([\omega]_n)\}_{n\geq0}$ is a strictly decreasing sequence. In particular, 
$D(\cdot,\cdot)$ is an ultrametric on $\Sigma$. i.e, for any $\omega,\tau,\eta\in\Sigma$,
\begin{equation*}
\max\{D(\omega,\eta),D(\eta,\tau)\}\geq D(\omega,\tau).
\end{equation*}
(2) Define $B_D(\omega,r)=\{\eta \in \Sigma:D(\omega,\eta)<r\}$ for any $\omega \in \Sigma$ and $r>0$.
 Then $B_D(\omega,r)=\Sigma([\omega]_n)$ if and only if $D([\omega]_n)<r\le D([\omega]_{n-1}).$
\end{Proposition}
The next result tells us when the harmonic measure ${\rm HARM}_{T}$
 satisfies the volume doubling property with respect to $D$
 ({\it i.e,} there exists a constant $c>0$ such that ${\rm HARM}_{T}(B_D(\omega,2r))\leq c{\rm HARM}_{T}(B_D(\omega,r))$
 for any $r>0$, and $\omega\in\Sigma$), which is a critical assumption for the heat kernel estimates in \cite{K}. 
\begin{Theorem}\cite[Theorem 6.5, Proposition 6.6.]{K}\label{vd}
\begin{itemize}
\item
The harmonic measure ${\rm HARM}_{T}$ has the volume
 doubling property with respect to $D$ if and only if the following conditions {\bf (EL)} and {\bf (D)} hold.
 \begin{description}
\item[{\bf (EL)}]: There exists $c_1\in(0,1)$ such that $c_1\le {\rm HARM}_{T}(\Sigma(x))/{\rm HARM}_{T}(\Sigma(\pi(x)))$
 for any $x\in T\backslash\{o\}.$
\item[{\bf (D)}]: There exist $m\geq1$ and $\theta \in(0,1)$ such that $D([\omega]_{n+m})\le \theta D([\omega]_n)$
 for any $n\geq0$ and $\omega \in \Sigma.$
 \end{description}
 \item The condition {\bf (EL)} implies that $\sup_{x\in T}\#S(x)<\infty$. 
 \end{itemize}
 Notice that the condition {\bf (EL)} fails
 if there exists $x_0\in T\setminus\{o\}$ such that $\#S(x_0)=1$.
 Thus, ${\rm HARM}_{T}$ does not satisfy the volume doubling property with respect to $D$ when
 either $\#S(x_0)=1$ for some $x_0\in T\setminus\{o\}$ or $\sup_{x\in T}\#S(x)=\infty$.
\end{Theorem}
\begin{Remark}\label{vdd}
 For $x\in T\setminus\{o\}$, we have 
 \begin{align}\label{sigma1}
 \Sigma(x)=\{\omega\in \Sigma\ ;\ N(x,\omega)\geq h(x)\}=\{\omega\in \Sigma\ ;\ d(x,\omega)\leq e^{-h(x)}\},
 \end{align}
 and
 \begin{align}\label{sigma2}
 \Sigma(\pi(x))=\{\omega\in \Sigma\ ;\ d(x,\omega)\leq e\cdot e^{-h(x)}\}.
 \end{align}
 Assume that we have the volume doubling property of ${\rm HARM}_T$ with respect to $d$, namely, 
 there exists a constant $C>1$ such that for any $\omega\in \Sigma$ and any $0<r<{\rm diam}(\Sigma,d)$, we have
 \begin{align*}
 {\rm HARM}_T(B_d(\omega,r))\leq C{\rm HARM}_T(B_d(\omega,2r)).
 \end{align*}
 Then by  (\ref{sigma1}) and (\ref{sigma2}),  for any $x\in T\setminus\{o\}$ we have 
 \begin{align*}
 \frac{1}{C}\leq\frac{{\rm HARM}_T(\Sigma(x))}{{\rm HARM}_T(\Sigma(\pi(x)))},
 \end{align*}
 which imply the condition {\bf(EL)} in Theorem \ref{vd} since $1/C\in(0,1)$.
 Thus, ${\rm HARM}_{T}$ does not satisfy the volume doubling property with respect to $d$ when
 either $\#S(x_0)=1$ for some $x_0\in T\setminus\{o\}$ or $\sup_{x\in T}\#S(x)=\infty$.
 \end{Remark}
 In \cite[Section 7]{K}, Kigami gives the following expression of the heat kernel
 associated with the regular Dirichlet form $(\mathcal{E}_{\Sigma},\mathcal{F}_{\Sigma})$ by using an eigenfunction expansion. 
\begin{eqnarray}\label{expression}\nonumber
p_t(\omega,\eta)&=&\sum_{n\geq0}\frac{\exp\left(-t/D\left([\omega]_{n-1}\right)\right)-\exp(-t/D([\omega]_n))}{{\rm HARM}_{T}(\Sigma([\omega]_n))}
\mathbf{1}_{\Sigma([\omega]_n)}(\eta)\\
&=&\begin{cases}
1+\sum_{n=0}^{\infty}\left(\dfrac{1}{{\rm HARM}_{T}(\Sigma([\omega]_{n+1}))}-\dfrac{1}{{\rm HARM}_{T}(\Sigma([\omega]_{n}))}\right)
\exp(-t/D([\omega]_n))\ \ {\rm if}\ \omega=\eta \\
\sum_{n=0}^{N(\omega,\eta)}\dfrac{1}{{\rm HARM}_{T}(\Sigma([\omega]_n))}
\left(\exp(-t/D([\omega]_{n-1}))-\exp(-t/D([\omega]_n)\right)
\ \ \ \ \ \ {\rm if}\ \omega \neq \eta, 
\end{cases}
\end{eqnarray}
 with the convention $1/D([\omega]_{-1})=0.$
 If we allow $\infty$ as a value, $p_t(\omega,\eta)$ is well-defined on $(0,\infty)\times \Sigma^2.$
 Note that we have $p_t(\omega,\eta)=p_t(\eta,\omega)$ and $p_t(\omega,\omega)\geq p_t(\omega,\eta)$
 for any $\omega,\eta\in\Sigma$ from the above expression.
 In fact, the heat kernel $p_t(\omega,\eta)$ which is given above is shown to be the transition density
 of the Hunt process associated with the regular Dirichlet form 
 $(\mathcal{E}_{\Sigma},\mathcal{F}_{\Sigma})$ under suitable assumptions.

\begin{Theorem}\label{semigroup}\cite[Proposition 7.2,Theorem 7.3.]{K} 
Assume that $\lim_{n \to \infty}D([\omega]_n)=0$ for any
 $\omega \in \Sigma.$
  Then,
\begin{equation*}
\int_{\Sigma}p_t(\omega,\eta)d{\rm HARM}_{T}(\eta)=1,\ \ and\ \ \int_{\Sigma}p_t(\omega,\xi)p_s(\xi,\eta)d{\rm HARM}_{T}(\xi)=p_{t+s}(\omega,\eta),
\end{equation*}
for any $\omega,\eta \in \Sigma$ with $\omega \neq \eta$ and any $t,s>0$.
 Moreover, there exists a Hunt process $(\{X_t\}_{t>0},\{P_{\omega}\}_{\omega\in\Sigma})$ on $\Sigma$ whose transition density is $p_t(\omega,\eta)$ i.e.
\begin{equation}
E_{\omega}(u(X_t))=\int_{\Sigma}p_t(\omega,\eta)u(\eta)d{\rm HARM}_{T}(\eta),
\end{equation}
for any $\omega \in \Sigma$ and any Borel measurable function $u:\Sigma\rightarrow\mathbb{R}$,
 where $E_{\omega}(\cdot)$ is the expectation with respect to $P_{\omega}$.
\end{Theorem}
\begin{Remark}
Since it is shown in \cite[Theorem 2.7]{K} that $D(x)=G(x,o)/C(o)$, the assumption in the above theorem is equivalent to
 the symmetrized Green function vanishing at infinity.
\end{Remark}
By the above theorem, if $\lim_{n \to \infty}D([\omega]_n)=0$\ for\ any\ $\omega \in \Sigma,$ then $p_tu(\omega)=T_tu(\omega)$ for ${\rm HARM}_{T}$-a.e.\ $\omega \in \Sigma,$ where $\{T_t\}_{t>0}$ is the strongly continuous semigroup on $L^2({\Sigma,{\rm HARM}_{T}})$ associated with the Dirichlet form $(\mathcal{E}_{\Sigma},\mathcal{F}_{\Sigma})$ on $L^2(\Sigma,{\rm HARM}_{T}).$ 

We will introduce the heat kernel estimates given in \cite[Proposition 7.5]{K}. First, the following estimate
 is shown without any further assumptions. 
\begin{Proposition}\label{HK}\cite[Proposition 7.5.]{K}\vspace{2mm}\\
(1) For any $\omega \in \Sigma$, and any $t>0$,
\begin{equation}p_t(\omega,\omega)\geq\frac{1}{e}\cdot\frac{1}{{\rm HARM}_{T}(B_D(\omega,t))}.\nonumber
\end{equation}
(2) If $0<t\le D(\omega,\eta)$, then
\begin{equation}p_t(\omega,\eta)\le \frac{t}{D(\omega,\eta){\rm HARM}_{T}(\Sigma([\omega,\eta]))}.\nonumber \end{equation}
\end{Proposition}

In \cite[Theorem 7.6.]{K}, the following two-sided estimates of $p_t(\omega,\eta)$ and the estimates
 of mean displacement
 are proved under
 the assumption of the volume doubling property of ${\rm HARM}_{T}$.
 Note that under the volume doubling property of ${\rm HARM}_{T}$,
 we have $\lim_{n \to \infty}D([\omega]_n)=0$ for any $\omega \in \Sigma$ by the condition {\bf(D)}
 in Theorem \ref{vd}.
 In the following, if $f$ and $g$ are two functions defined on a set $U$,
 $f\asymp g$ means there exists $C>0$ such that $C^{-1}f(x)\le g(x)\le Cf(x)$ for all $x\in U$.

\begin{Theorem}\label{HKE}\cite[Theorem 7.6, Corollary 7.9.]{K} Suppose ${\rm HARM}_{T}$ has the volume doubling property
with respect to $D$. Then, the following results hold.\\
(1) The heat kernel $p_t(\omega,\eta)$ is continuous on $(0,\infty)\times\Sigma^2$. Define
\begin{equation}
q_t(\omega,\eta)
=\begin{cases}\dfrac{t}{D(\omega,\eta){\rm HARM}_{T}(\Sigma([\omega,\eta]))}\ \ &${\rm if}$\ \ \ 0<t\le D(\omega,\eta),\\
\dfrac{1}{{\rm HARM}_{T}(B_D(\omega,t))}\ \ \ \ \ \ \ &${\rm if}$\ \ \ t>D(\omega,\eta).
\end{cases}
\end{equation}
Then,
\begin{equation*}
 p_t(\omega,\eta)\asymp q_t(\omega,\eta)\ {\rm on}\ (0,\infty)\times\Sigma^2.
 \end{equation*}
(2) For any $\omega \in \Sigma$ and any $t\in(0,1]$,
\begin{equation*}
E_{\omega}[D(\omega,X_t)^{\gamma}]\asymp\begin{cases}t\ \ \ \ \ \ \ \ \ \ \ \ \ \ \ \ \ ${\rm if}$\ \ \gamma>1,\\
t(|\log t|+1)\ \ \ ${\rm if}$\ \ \gamma=1,\\
t^{\gamma}\ \ \ \ \ \ \ \ \ \ \ \ \ \ \ \ ${\rm if}$\ \ 0<\gamma<1.
\end{cases}
\end{equation*}

\end{Theorem}

\section{Electric networks on Galton-Watson trees}

\subsection{The asymptotics of the effective resistance along infinite rays}

In the previous section, we have presented results on the construction and properties of jump processes
 on the boundaries of  deterministic trees studied in \cite{K}. In this section, we will consider random trees instead of
 deterministic trees, in particular Galton-Watson trees. First we will introduce the preliminary results
 of electric networks and corresponding random walks on Galton Watson trees which will
 be important for our study.\par
 Let $\mathcal{T}$ be a rooted Galton-Watson tree with offspring distribution $\{p_k\}_{k\geq0}$.
 Note that 
 \begin{equation*}
 \left\{\xi_n\}_{n\geq0}=\{\#\{v\in\mathcal{T}: h(v)=n\}\right\}_{n\geq0}
 \end{equation*}
 is a Galton-Watson process with offspring distribution $\{p_k\}_{k\geq0}$.
 Throughout this paper, we will assume the following condition on $\{p_k\}_{k\geq0}$.
\begin{equation}\label{off}
\begin{cases}
p_0=0,\ p_1\neq1,\\
 1<m<\infty,\ \ {\rm where}\ m:=\sum_{k=0}^{\infty}kp_k.
\end{cases}
\end{equation}
Recall that under the above assumptions, $\mathcal{T}$ can be regarded as a $\mathbb{T}$-valued random variable,
 where $\mathbb{T}:=\{T\ ;\ T\ {\it is\ an\ infinite\ rooted\ tree}\}$,
 and we will denote its distribution by $\mathbb{P}_{{\rm GW}}$.
Moreover, $\mathcal{T}(v)$ is an infinite tree for all $v\in\mathcal{T}$
 with $\mathbb{P}_{{\rm GW}}$-probability 1, so the extinction event has $\mathbb{P}_{{\rm GW}}$-probability 0.\par
 For an infinite tree $T$ and $\lambda>0$, we consider the $\lambda$-{\it biased\ random\ walk}
 $\{Z_n^{\lambda}\}_{n\geq0}$ on $T$
 under the probability measure $P_{\lambda}^{T}$.
 Let the initial state $Z_0$ to be $o$ unless otherwise stated. This is equivalent to considering the
 weighted graph $(T,C^{\lambda}(T))$, where the conductance of an edge connecting vertices at level $n$ and $n+1$
 is $\lambda^{-n}$. For $0<\lambda<m,$ it is proved in \cite{L} that
 the $\lambda$-biased random walk on the Galton-Watson tree
 $\mathcal{T}$ is transient with $\mathbb{P}_{{\rm GW}}$-probability 1.
 Thus for $0<\lambda<m$, we have the harmonic measure of
 $\{Z_n^{\lambda}\}_{n\geq0}$,
 which is a probability measure on $\Sigma$. 
 In what follows,
 when the $\lambda$-biased random walk on an infinite rooted tree $T$ is transient, we will denote the harmonic
 measure of the random walk staring from a vertex $x$ by ${\rm HARM}^{\lambda}_{T,x}$.
 Again, we set ${\rm HARM}^{\lambda}_{T}={\rm HARM}^{\lambda}_{T,o}$. 
 Note that by the condition {\bf (EL)} in Theorem \ref{vd}, ${\rm HARM}_{\mathcal{T}}^{\lambda}$ does not have
 the volume doubling
 property with respect to the intrinsic metric $D$ with $\mathbb{P}_{{\rm GW}}$-probability 1
 when either $p_1>0$ or $\sup\{n: p_n>0\}=+\infty.$
 Moreover, by Remark \ref{vdd}, a similar claim holds for the metric $d$.  
 Since the purpose of the paper is to obtain estimates of the heat kernels without the volume doubling property
 of the harmonic measure, estimates of the harmonic measure and the effective resistance will be important.
 The following estimates of the harmonic measure shown in \cite[Theorem 1.1]{LPP1} \cite[Theorem 5.1]{LPP2}
 play an important role in what follows.
\begin{Theorem}\label{HM}\ For $0<\lambda<m,$ the following results hold.\\
(1) There exists a deterministic constant $0<\beta_\lambda<\log m$ such that 
 $\beta_{\lambda}=\dim{{\rm HARM}_{\mathcal{T}}^{\lambda}}$ $\mathbb{P}_{{\rm GW}}$-$a.s.$\\ 
(2) Define
\begin{equation}
\Sigma_{1}:=\{\omega \in \Sigma: -\lim_{n \to \infty}\frac{\log{{\rm HARM}_{\mathcal{T}}^{\lambda}(B_d(\omega,e^{-n}))}}{n}=-\lim_{n \to \infty}\frac{\log{{\rm HARM}_{\mathcal{T}}^{\lambda}(\Sigma([\omega]_n))}}{n}=\beta_{\lambda}\},\nonumber
\end{equation}
then ${\rm HARM}_{\mathcal{T}}^{\lambda}(\Sigma_{1})=1$\ $\mathbb{P}_{{\rm GW}}$-a.s.
\end{Theorem}

Next, we will investigate the behavior of the effective resistance.
 Before giving the statement, we need some preparations. In \cite{LPP1,LPP2}, Markov chains on
  ``the space of trees'' are studied and, in particular, Markov chains associated with harmonic flows are utilized
 to study the behavior of harmonic measures of $\lambda$-biased random walks on Galton-Watson trees.

\begin{Definition} 
 (1) For a tree $T\in\mathbb{T}$ and $v\in T$, define $R^{\lambda}(T,v)=R^{\lambda}(v)$ as the effective resistance of
 $(T(v),C^{\lambda}|_{T(v)})$ from the vertex $v$ to infinity. Define the effective conductance $EC^{\lambda}(T(v))$ by
 $EC^{\lambda}(T(v)):=(R^{\lambda}(v))^{-1}$. \\
 (2) For a tree $T\in\mathbb{T}$, a nonnegative function $\theta$ on $T$ is called a flow on $T$ if for all $x\in T$,
 $\theta(x)=\sum_{y:\pi(y)=x}\theta(y)$. Note that flows on $T$ are in one-to-one correspondence
 with positive finite Borel measures $\mu$ on $\Sigma(T)$ by
 $\theta(x)=\mu(\Sigma(x))$.\\
 (3) Define $\mathbb{F}$:=\{$\theta$; $\theta$ is a flow on $T$ for some $T\in\mathbb{T}$\}.
 A Borel map $\Theta:\mathbb{T}\rightarrow\mathbb{F}$ is called a flow rule if
 for any $T\in\mathbb{T}$, $\Theta(T)$ is a flow on $T$,
 and for any $x\in T$ with $|x|=1$ and $\Theta(T)(x)>0$, we have $\Theta(T)(y)/\Theta(T)(x)=\Theta(T(x))(y)$ for $y\in T(x)$.
\end{Definition}
For a given flow rule $\Theta$, we can associate a Markov chain on $\mathbb{T}$ in the following way:
 for a flow rule $\Theta$, define transition probabilities $p_{\Theta}$ by
 $p_{\Theta}(T,T(x)):=\Theta(T)(x)$ for $T\in\mathbb{T}$ and $x\in T$ with $|x|=1$.
 A path of this Markov chain on $\mathbb{T}$ is naturally identified with an element of
 $\mathbb{T}_{{\rm ray}}:=\{(T,\omega);\ T\in\mathbb{T},\ \omega\in\Sigma(T)\}$. Define the shift operator $S$ on
  $\mathbb{T}_{{\rm ray}}$ by 
  \begin{align}\label{Sdefn}
  S\big((T,\omega)\big):=(T([\omega]_1),\tau(\omega)),
  \end{align}
 where $\tau$ is the shift operator on $\Sigma$.
 A measure $\mu$ on $\mathbb{T}$ is called {\it $\Theta$-stationary} if for any Borel subset $A\subseteq\mathbb{T}$,
 \begin{eqnarray*}
\mu(A)=\int\sum_{T'\in A}p_{\Theta}(T,T')d\mu(T')=\int\sum_{|x|=1,T(x)\in A}\Theta(T)(x)d\mu(T).
\end{eqnarray*}
When the $\lambda$-biased random walk on $T\in\mathbb{T}$ is transient, its path converges almost surely to
 a random element of $\Sigma$, and its distribution is ${\rm HARM}^{\lambda}_T$.
 Define ${\rm HARM}^{\lambda}:\mathbb{T}\rightarrow\mathbb{F}$ by ${\rm HARM}^{\lambda}(T)(x):={\rm HARM}^{\lambda}_T(\Sigma(x))$
 for $x\in T$ with $|x|=1$. It is obvious that this defines a flow rule. In the rest of this section, we will use the following
 results obtained by \cite{Lin} and \cite{Rou} independently, which give the explicit expression of the
 ${\rm HARM}^{\lambda}$-stationary probability measure.
\begin{Definition}
For a tree $T\in\mathbb{T}$ rooted at $o$, define $\tilde{T}$ as the tree obtained by drawing an extra edge between
 $o$ and an extra adjacent vertex $\tilde{o}$, which is the root of $\tilde{T}$. Let
 \begin{align*}
 \alpha^{\lambda}(T)&:=P_{\lambda}^{\tilde{T}}(Z_n^{\lambda}\neq\tilde{o}\ {\it for\ all}\ n\geq1\ |\ Z_0^{\lambda}=o).
 \end{align*}
 be the probability that the $\lambda$-biased random walk on $\tilde{T}$ starting at $o$ never visits $\tilde{o}$.
\end{Definition}
\begin{Proposition}\label{inv}
For $0<\lambda<m$, define $\theta^{\lambda}:\mathbb{T}\rightarrow\mathbb{R}$ and $h^{\lambda}\in\mathbb{R}_{\geq0}$
 by
\begin{align*}
\theta^{\lambda}(T)&:=\int\frac{\alpha^{\lambda}(T')EC^{\lambda}(T)}
{\lambda-1+\alpha^{\lambda}(T')+EC^{\lambda}(T)}d\mathbb{P}_{{\rm GW}}(T'),\\
h^{\lambda}&:=\int_{\mathbb{T}}\theta^{\lambda}(T)d\mathbb{P}_{{\rm GW}}(T).
\end{align*}
Then, the measure
 $d\mu^{\lambda}_{{\rm HARM}}(T):=(h^{\lambda})^{-1}\cdot\theta^{\lambda}(T)d\mathbb{P}_{{\rm GW}}(T)$
 is the unique 
${\rm HARM}^{\lambda}$-stationary probability measure which is mutually absolutely continuous
 with respect to $\mathbb{P}_{{\rm GW}}$.
 Moreover, the ${\rm HARM}^{\lambda}$-Markov chain with initial distribution $\mu_{{\rm HARM}}^{\lambda}$ is ergodic.
\end{Proposition}
\begin{proof}
The explicit formula for $\mu_{{\rm HARM}}^{\lambda}$ is given in \cite[Lemma 5]{Lin} and \cite[Theorem 4.1]{Rou}.
 The ergodicity of the ${\rm HARM}^{\lambda}$-Markov chain is proved in \cite{LPP1,LPP2}.
\end{proof}

The following result will be important for the proof of our main result.
\begin{Proposition}\label{Resistance}Under the assumption (\ref{off}),
 for $0<\lambda<m$, the following holds $\mathbb{P}_{{\rm GW}}$-a.s.
\begin{equation}\label{erg}
\lim_{n\to\infty}\frac{1}{n}\log R^{\lambda}([\omega]_n)=
\log\lambda,\ \ {\rm HARM}_{\mathcal{T}}^{\lambda}\mathchar`-a.e.\ \omega.
\end{equation}
 In other words, if define
\begin{equation*}
\Sigma_{2}:=\{\omega\in\Sigma; \lim_{n\to\infty}\frac{1}{n}\log R^{\lambda}([\omega]_n)=\log\lambda\},
\end{equation*}
then ${\rm HARM}_{\mathcal{T}}^{\lambda}(\Sigma_{2})=1$\ $\mathbb{P}_{{\rm GW}}$-a.s.
\end{Proposition}

 Before giving the proof of Proposition \ref{Resistance}, we prove the following moment estimates of the effective resistance,
 which are of independent interest.

\begin{Lemma} \label{rrr}(1) For $0<\lambda<1$, $R^{\lambda}(o)\leq\frac{1}{1-\lambda}$. In particular,
 for $x\in\mathcal{T}_k$ we have $R^{\lambda}(x)\leq \lambda^{k}\cdot\frac{1}{1-\lambda}\ \mathbb{P}_{{\rm GW}}$-$a.s.$\\
(2) For $\lambda=1$, there exists a constant $b>0$ such that
 $\mathbb{P}_{{\rm GW}}(R^1(o)>n)\leq e^{-bn}$ for any $n\in\mathbb{N}$.\\
(3) For $1<\lambda<m$, if $p_1>0, $ there exists a constant $c>0$ such that
 $\mathbb{P}_{{\rm GW}}(R^{\lambda}(o)>n)\leq cn^{\log p_1/\log\lambda}$ for any $n\in\mathbb{N}$.
 If $p_1=0$, for any $\alpha>0$ there exists a constant $c_{\alpha}>0$ such that
 $\mathbb{P}_{{\rm GW}}(R^{\lambda}(o)>n)\leq c_{\alpha}n^{-\alpha}$ for any $n\in\mathbb{N}$.\\
(4) For $0<\lambda<m$, if\ $\sum_{k\geq1}k^{\alpha}p_k<\infty$ for $\alpha\geq1$,
 we have $\mathbb{E}_{{\rm GW}}[(R^{\lambda}(o))^{-\alpha}]<\infty$.
\end{Lemma}
\begin{proof}
The first claim immediately follows from Rayleigh's monotonicity principle.
 We will prove the rest.
 Remark that the following argument heavily relies on that in \cite[Section 4]{LPP3}.
 First, we assume $p_1>0$.
 For $T\in\mathbb{T}$, it is easy to show that
 \begin{align}\label{gamma}
 \alpha^{\lambda}(T)=\frac{\sum_{v\in S(o)}\alpha^{\lambda}(T(v))}{\lambda+\sum_{v\in S(o)}
 \alpha^{\lambda}(T(v))}=\frac{1}{1+\lambda R^{\lambda}(o)}.
 \end{align}
 Define 
 \begin{align*}
 F_{\lambda}(s):=\mathbb{P}_{{\rm GW}}(\alpha^{\lambda}(T)\leq s)=
 \mathbb{P}_{{\rm GW}}\left(R^{\lambda}(o)\geq\frac{1-s}{s\lambda}\right),
\end{align*}
and denote the $k$-th fold convolution of $F_{\lambda}$ by $F_{\lambda}^{*k}$. Then by combining the branching property of
 Galton-Watson trees and (\ref{gamma}), we obtain
\begin{align*}
F_{\lambda}(s)=
\begin{cases}
\sum_{l}p_lF_{\lambda}^{*l}(\frac{s\lambda}{1-s})\ \ \ $if$\ s\in(0,1),\\
0\  \  \  \  \  \ \ \ \ \ \ \ \ \ \ \  \ \ \ \ $if$\ s\leq0 ,                           \\
1\  \  \  \  \ \ \ \ \ \ \ \ \ \ \ \ \ \ \ \ $if$\ s\geq1.
\end{cases}
\end{align*}
By this expression, we get 
\begin{align*}
\mathbb{P}_{{\rm GW}}(R^{\lambda}(o)\geq n)=F_{\lambda}\left(\frac{1}{1+\lambda n}\right)
=\sum_lp_lF_{\lambda}^{*l}\left(\frac{1}{n}\right)
\leq\sum_lp_lF_{\lambda}\left(\frac{1}{n}\right)^l
=\sum_lp_l\mathbb{P}_{{\rm GW}}\left(R^{\lambda}(o)\geq\frac{n-1}{\lambda}\right)^l.
\end{align*}
So if we define $G(x):=\sum_lp_lx^l$, we have 
\begin{align*}
\mathbb{P}_{{\rm GW}}(R^{\lambda}(o)\geq n)\leq
 G\left(\mathbb{P}_{{\rm GW}}\left(R^{\lambda}(o)\geq\frac{n-1}{\lambda}\right)\right)\leq
G\circ G\left(\mathbb{P}_{{\rm GW}}\left(R^{\lambda}(o)\geq\frac{n-\lambda-1}{\lambda^2}\right)\right)\leq....
\end{align*}
 Since $R^{\lambda}(o)<\infty\ a.s.$ for $0<\lambda<m$,
 we have $\mathbb{P}_{{\rm GW}}(R^{\lambda}(o)\geq N)<1/2$ for $N$ large enough.
 Hence we obtain that $\mathbb{P}_{{\rm GW}}(R^{\lambda}(o)\geq n)\leq G^{(l)}(1/2)$ for $l\in\mathbb{N}$
 satisfying 
 \begin{align*}
 \frac{n-\sum_{j=0}^{l-1}\lambda^j}{\lambda^l}>N.
 \end{align*}
  Note that by \cite[Chapter 1, Section 11]{AN},
 there exists a positive function $Q(s)$ on $[0,1]$ such that 
 \begin{align*}
 \lim_{n \to \infty}p_1^{-n}G^{(n)}(s)=Q(s)\ \ {\rm for\ any}\ s\in[0,1].
 \end{align*}
 This implies that when $\lambda=1$ there exists a constant $b>0$ such that
 \begin{align*}
 \mathbb{P}_{{\rm GW}}(R^1(o)>n)\leq e^{-bn},
 \end{align*}
 and when $1<\lambda<m$ there exists a constant $c>0$ such that
 \begin{align*}
 \mathbb{P}_{{\rm GW}}(R^{\lambda}(o)>n)\leq cn^{\log p_1/\log\lambda}.
 \end{align*}
 These estimates imply the second claim, and the third claim in the case of $p_1>0$.\par
  When $1<\lambda<m$ and $p_1=0$,
 for $\varepsilon>0$
 we define another offspring distribution $\{q^{\varepsilon}_k\}_{k\geq0}$ by 
 $q^{\varepsilon}_0:=0$,
 $q^{\varepsilon}_1:=\varepsilon$, $q^{\varepsilon}_L:=p_L-\varepsilon$ and
 $q^{\varepsilon}_n:=p_n$ for $n>L$, where $L:=\inf\{k:p_k>0\}$. If we take $\varepsilon>0$ sufficiently small such that
 $\varepsilon<p_L$ and $\lambda<\sum_{k\geq1}kq^{\varepsilon}_k$, it is easy to see that we can define the Galton-Watson
 tree with offspring distribution $\{q^{\varepsilon}_k\}_{\geq0}$ on the probability space $(\mathbb{T},\mathbb{P}_{{\rm GW}})$
 in such a way that $R^{\lambda}(o)$ is stochastically dominated by the effective resistance of $\lambda$-biased
 random walk on the Galton-Watson tree with offspring distribution $\{q^{\varepsilon}_k\}_{\geq0}$.
 Hence, we get the third claim. The fourth claim is immediate from the equality
\begin{align*}
\frac{1}{R^{\lambda}(o)}=\sum_{v\in S(o)}\frac{1}{1+R^{\lambda}(v)}\leq\# S(o),
\end{align*}
 which follows from the parallel and serial laws of basic electric network theory.
\end{proof}

{\it Proof of Proposition \ref{Resistance}.} We shall verify below that the
 function $f^{\lambda}:\mathbb{T}_{{\rm ray}}\rightarrow\mathbb{R}$ defined by
\begin{equation*}
f^{\lambda}((T,\omega)):=\log\frac{R^{\lambda}([\omega]_1)}{R^{\lambda}(o)},
\end{equation*}
 is integrable 
 with respect to $\mu_{{\rm HARM}}^{\lambda}\times{\rm HARM}_T^{\lambda}$.
 Then Proposition \ref{inv} and  Birkhoff's ergodic theorem imply that for
 $\mu_{{\rm HARM}}^{\lambda}\times
{\rm HARM}^{\lambda}_T$-almost every $(T,\omega)$,
\begin{eqnarray*}
\lim_{n\to\infty}\frac{1}{n}\log R^{\lambda}([\omega]_n)=\lim_{n\to\infty}\frac{1}{n}\sum_{k=0}^nS^kf^{\lambda}(T,\omega)
=\int_{\mathbb{T}_{{\rm ray}}}\log\frac{R^{\lambda}([\xi]_1)}{R^{\lambda}(o)}
d\mu^{\lambda}_{{\rm HARM}}\times{\rm HARM}_T^{\lambda}(T,\xi),
\end{eqnarray*}
where $S$ is the shift operator on $\mathbb{T}_{{\rm ray}}$ defined in (\ref{Sdefn}).
By stationarity of $\mu_{{\rm HARM}}^{\lambda}$, we get
\begin{align*}
\int_{\mathbb{T}_{{\rm ray}}}\log\frac{R^{\lambda}([\xi]_1)}{R^{\lambda}(o)}d\mu^{\lambda}_{{\rm HARM}}\times{\rm HARM}_T^{\lambda}(T,\xi)=\log\lambda,
\end{align*}
which yields the claim of Proposition \ref{Resistance},
 since $\mathbb{P}_{{\rm GW}}\ll \mu_{{\rm HARM}}^{\lambda}$ by Proposition \ref{inv}. The integrability
 of $f^{\lambda}$ with respect to $\mu_{{\rm HARM}}^{\lambda}\times{\rm HARM}_T^{\lambda}$ can be proved as follows:
 since
 \begin{align*}
 &\int_{\mathbb{T}_{{\rm ray}}}\left|\log\frac{R^{\lambda}([\xi]_1)}{R^{\lambda}(o)}\right|d\mu^{\lambda}_{{\rm HARM}}\times{\rm HARM}_T^{\lambda}(T,\xi)
 \leq \int_{\mathbb{T}_{{\rm ray}}}\big(\left|\log R^{\lambda}([\xi]_1)\right|+\left|\log R^{\lambda}(o)\right|\big)d\mu^{\lambda}_{{\rm HARM}}\times{\rm HARM}_T^{\lambda}(T,\xi),\\
 &\leq|\log\lambda|+2\int_{\mathbb{T}_{{\rm ray}}}\left|\log R^{\lambda}(o)\right|d\mu^{\lambda}_{{\rm HARM}}\times{\rm HARM}_T^{\lambda}(T,\xi),
 \end{align*}
 it suffices to prove the integrability of $\log R^{\lambda}(o)$ with respect to $\mu_{{\rm HARM}}^{\lambda}$ because
 the random variable $\log R^{\lambda}(o)$ does not depend on $\xi$.
 It is shown in the proof of \cite[Lemma 5]{Lin} that $0<\theta^{\lambda}(x)<1$ for any $x>0$, hence we only need to prove
 $\log R^{\lambda}(o)\in L^1(\mathbb{T},\mathbb{P}_{{\rm GW}})$. Since
 \begin{align*}
 \mathbb{P}_{{\rm GW}}(|\log R^{\lambda}(o)|\geq n)=\mathbb{P}_{{\rm GW}}(R^{\lambda}(o)\geq e^n)
 +\mathbb{P}_{{\rm GW}}(R^{\lambda}(o)\leq e^{-n}),
 \end{align*}
 we get the desired result by Lemma \ref{rrr}. \qed\\

 As a corollary of Proposition \ref{Resistance}, we obtain the following lower bound of $\beta_{\lambda}$.
 
 \begin{Corollary} \label{>}
For $0<\lambda<m$, we have $\beta_{\lambda}>0\vee\log\lambda$.
\end{Corollary}

\begin{proof}
It is proved in \cite{LPP1,LPP2} that for $0<\lambda<m$,
\begin{align*}
\beta_{\lambda}=\int_{\mathbb{T}_{{\rm ray}}}\log\frac{1}{{\rm HARM}^{\lambda}_{T}([\xi]_1)}
d\mu^{\lambda}_{{\rm HARM}}\times{\rm HARM}_T^{\lambda}(T,\xi).
\end{align*}
By using \cite[Theorem 3.8.]{K}, we have for $0<\lambda<m$,
\begin{align*} 
\beta_{\lambda}-\log\lambda
=\int_{\mathbb{T}_{{\rm ray}}}\log\left(\frac{1+R^{\lambda}([\xi]_1)}{R^{\lambda}(o)}\cdot
\frac{R^{\lambda}(o)}{R^{\lambda}([\xi]_1)}\right)d\mu^{\lambda}_{{\rm HARM}}\times{\rm HARM}_T^{\lambda}(T,\xi)>0,
\end{align*}
and we get the conclusion.
\end{proof}

\begin{Remark}
\begin{itemize}
\item
In \cite[Section 5]{Lin}, a better lower bound of $\beta_{\lambda}$ is obtained. 
\item
Corollary \ref{>} and the fact that $\beta_{\lambda}<\log m$ imply that 
$\lim_{\lambda\nearrow m}\beta_{\lambda}=\log m$, where $\log m$ is the Hausdorff dimension of the boundary
{\rm (}$\Sigma,d${\rm )} $\mathbb{P}_{{\rm GW}}$-a.s. This is a part of the results stated in \cite[Theorem 1]{Lin}.
 \end{itemize}
\end{Remark}

\section{Proof of the main theorems}
In this section, we construct jump processes on the boundaries of Galton-Watson trees and give
 the proof of the short time
 log-asymptotic of the heat kernels and the estimates of mean displacements.
 For $0<\lambda<m$, let $D^{\lambda}$ be the intrinsic metric of $(\mathcal{T},C^{\lambda})$ on $\Sigma$,
 $(\mathcal{E}^{\lambda}_{\Sigma},\mathcal{F}^{\lambda}_{\Sigma})$ be the Dirichlet form on
 $L^2(\Sigma,{\rm HARM}_{\mathcal{T}}^{\lambda})$
 which corresponds to the $\lambda$-biased random walk on $\mathcal{T}$ and
 $p^{\lambda}_t(\omega,\eta)$ be the heat kernel associated with the Dirichlet form
 $(\mathcal{E}^{\lambda}_{\Sigma},\mathcal{F}^{\lambda}_{\Sigma})$.
 Define $\Sigma_*:=\Sigma_1\cap\Sigma_{2}$. Note that ${\rm HARM}_{\mathcal{T}}^{\lambda}(\Sigma_*)=1$
 $\mathbb{P}_{{\rm GW}}$-$a.s.$
\begin{Theorem}For $0<\lambda<m$, the\ following\ results\ hold\ $\mathbb{P}_{{\rm GW}}$-$a.s$.
\begin{eqnarray*}
\int_{\Sigma}p_t^{\lambda}(\omega,\tau)d{\rm HARM}_{\mathcal{T}}^{\lambda}(\tau)=1,\ \  {\it and}\ \ 
\int_{\Sigma}p_t^{\lambda}(\omega,\xi)p_s^{\lambda}(\xi,\eta)d{\rm HARM}_{\mathcal{T}}^{\lambda}(\xi)
=p_{t+s}^{\lambda}(\omega,\eta),
\end{eqnarray*}
for\ any\ $\omega,\eta \in \Sigma_{*}$\ with\ $\omega \neq \eta$ and\ any\ $t,s>0$.
 Moreover, there exists a Hunt process $(\{X_t^{\lambda}\}_{t>0},\{P_{\omega}^{\lambda}\}_{\omega\in\Sigma})$ on $\Sigma_{*}$ whose transition density is $p_t^{\lambda}(\omega,\eta)$ $\mathbb{P}_{{\rm GW}}$-$a.s.$
\end{Theorem}
\begin{proof} By Theorem \ref{HM}, Proposition \ref{Resistance} and Corollary \ref{>}, for $0<\lambda<m$,
 we have $\lim_{n\to\infty}D^{\lambda}([\omega]_n)=0$ for $\omega\in\Sigma_*$ $\mathbb{P}_{{\rm GW}}$-$a.s.$
 By using (\ref{expression}), a routine calculation yields the first statement.
 In order to prove the second statement, it suffices to show that for $0<\lambda<m$, 
 $p_t^{\lambda}(C(\Sigma_{*}))\subseteq C(\Sigma_{*})$, and $\| p_t^{\lambda}u-u \|_{\infty}\rightarrow0$
 as $t\to0$\ for\ any\ $u\in C(\Sigma_{*})$ $\mathbb{P}_{{\rm GW}}$-$a.s.$,
 where $C(\Sigma_{*})$ is the collection of continuous functions on $\Sigma_{*}$.
 It can be proved by a similar argument to that in \cite[Theorem 7.3, Lemma 7.4]{K}. 
\end{proof}

We now prove Theorem \ref{OND}. In the rest of this paper, we write
 $X_t^{\lambda}, E^{\lambda}_{\omega}, D^{\lambda},\ {\rm HARM}_{\mathcal{T}}^{\lambda}$, $R^{\lambda}(v)$
 and$\ p_t^{\lambda}$ as $X_t, E_{\omega}, D,\ {\rm HARM},\ R(v)$ and$\ p_t$ respectively.\\

{\it Proof of Theorem \ref{OND}.}
We first prove (\ref{A}). By Proposition \ref{HK}, for $\omega,\eta\in\Sigma$ with $\omega\neq\eta$,
 and $0<t\leq D(\omega,\eta)$ we have
\begin{align*}
p_t(\omega,\eta)\leq\frac{t}{D(\omega,\eta){\rm HARM}(\Sigma([\omega,\eta]))}.
\end{align*}
Recalling the convention $1/D([\omega]_{-1})=0$, by (\ref{expression}) and the first claim of Proposition \ref{D},
 we have
\begin{align*}
p_t(\omega,\eta)&=\sum_{n=0}^{N(\omega,\eta)}\dfrac{1}{{\rm HARM}_{T}(\Sigma([\omega]_n))}
\left(\exp(-t/D([\omega]_{n-1}))-\exp(-t/D([\omega]_n)\right)\\
&\geq \dfrac{1}{{\rm HARM}_{T}(\Sigma([\omega]_0))}
\left(\exp(-t/D([\omega]_{-1}))-\exp(-t/D([\omega]_0)\right)=1-\exp(-t/D([\omega]_0).
\end{align*}
Therefore, for any $\omega,\eta\in\Sigma$ with $\omega\neq\eta$, we get 
\begin{align*}
\lim_{t\to 0}\frac{\log p_t(\omega,\eta)}{\log t}=1.
\end{align*}
\par
We next prove (\ref{AA}).
It is sufficient to prove the claim for any $\omega\in\Sigma_{*}$.
 We will write $D_n=D([\omega]_n)$, $R_n=R([\omega]_n)$ and $ H_n={\rm HARM}(\Sigma([\omega]_n))$.
 Let $l=l(\omega,t)$ be the unique integer which satisfies $D_l < t \le D_{l-1}$.
 Then we have $\lim_{t \to +0} l(\omega,t)=\infty$ for\ all\ $\omega\in\Sigma_{*}.$
Note that by Proposition \ref{HK}, we have the following lower bound.
\begin{equation*}
p_t(\omega,\omega)\geq \frac{1}{e}\cdot\frac{1}{{\rm HARM}(B_D(\omega,t))}=\frac{1}{e}\cdot\frac{1}{ H_l}.
\end{equation*}
So we will prove the upper bound. Recall that by (\ref{expression}), we have 
\begin{align}\label{defn}
p_t(\omega,\omega)=1+\sum_{n=0}^{\infty} \left(\dfrac{1}{ H_{n+1}} - \dfrac{1}{ H_n}\right)
\exp(-t/D_n).
\end{align}
 We get 
\begin{eqnarray}\label{til_l}
1&+&\sum_{n=0}^{l-1} \left(\dfrac{1}{ H_{n+1}} - \dfrac{1}{ H_n}\right)\nonumber
\exp(-t/D_n) \le 1+\sum_{n=0}^{l-1} \left(\dfrac{1}{H_{n+1}} -\dfrac{1}{H_n}\right)\\
&=&\dfrac{1}{ H_l}
=\dfrac{1}{ {\rm HARM}(\Sigma([\omega]_l))}=\dfrac{1}{ {\rm HARM}(B_d(\omega,e^{-l}))}.
\end{eqnarray}
On the other hand, we have
\begin{eqnarray}\label{after_l}
\sum_{n=l}^{\infty} \left(\frac{1}{H_{n+1}} - \frac{1}{H_n}\right)\exp(-t/D_n)  
\le \sum_{n=l}^{\infty} \frac{1}{H_{n+1}}\exp(-t/D_n)\le
 \frac{1}{H_l} \sum_{n=l}^{\infty} \frac{H_l}{H_{n+1}} \exp\left(-\frac{D_l}{D_n}\right).
\end{eqnarray}
By Theorem 3.1, we have the following control of the volume. For all $\varepsilon>0$, there exists 
 a random integer $L=L(\omega,\varepsilon)$
 such that for all $n\geq L$,
\begin{equation*}
\exp \{-(\beta_{\lambda}+\varepsilon)n\} \le H_n \le \exp \{-(\beta_{\lambda}-\varepsilon)n\}.
\end{equation*}
For $t>0$ sufficiently small, we have $l\geq L$, and
\begin{align*}
\sum_{n=l}^{\infty} \frac{H_l}{H_{n+1}} \exp\left(-\frac{D_l}{D_n}\right)  
&\le C\sum_{n=l}^{\infty} \exp \{(\beta_{\lambda}+\varepsilon)(n-l)
+2\varepsilon l \}\cdot\exp \left[-C' \frac{R_l}{R_n}\exp\{(\beta_{\lambda}-\varepsilon)(n-l)-2\varepsilon l\}\right]  \\
&\le C\exp(2\varepsilon l) \sum_{k=0}^{\infty} \exp \{(\beta_{\lambda}+\varepsilon)k \} 
\exp\left[-\frac{C'}{\exp(2\varepsilon l)}\frac{R_l}{R_{l+k}}\exp \{(\beta_{\lambda}-\varepsilon)k \}\right], 
\end{align*}
where $C$ and $C'$ are constants which do not depend on $t$. 
By Proposition \ref{Resistance}, for all $\varepsilon>0$, there exists $K=K(\omega,\varepsilon)$ such that
 for all $n\geq K$,
\begin{align*} 
\exp\{(\log\lambda-\varepsilon)n\}\leq R_n\leq\exp\{(\log\lambda+\varepsilon)n\}.
\end{align*}
Hence, for $0<\lambda<m$, $\varepsilon >0$, $l\geq N\vee K$ and $t>0$ sufficiently small, 
\begin{align*}
\sum_{n=l}^{\infty} \frac{H_l}{H_{n+1}} \exp\left(-\frac{D_l}{D_n}\right)  
&\le C\exp(2\varepsilon l) \sum_{k=0}^{\infty} \exp\{(\beta_{\lambda}+\varepsilon)k\}
\cdot\exp \left[-\frac{C'}{\exp(2\varepsilon l)}\frac{\exp\{(\log\lambda-\varepsilon) l\}}
{\exp \{(\log\lambda+\varepsilon) (l+k) \}} \exp \{(\beta_{\lambda}-\varepsilon)k\}\right]  \\
&\le C\exp(2\varepsilon l) \sum_{k=0}^{\infty} \exp \{(\beta_{\lambda}+\varepsilon)k \}\exp \left[-\frac{C'}{\exp(4\varepsilon l)} \exp\{(\beta_{\lambda}-\log\lambda-2\varepsilon)k\}\right]  \\
&\le C\exp(2\varepsilon l) \sum_{k=0}^{\infty} \alpha^{Qk}\exp\left\{-\frac{C'}{\exp(4\varepsilon l)} \alpha^k \right\}, 
\end{align*}
where\  $\alpha=\exp(\beta_{\lambda}-\log\lambda-\varepsilon)>1$ and
 $Q:=\frac{\beta_{\lambda}+\varepsilon}{\beta_{\lambda}-\log\lambda-\varepsilon}+1$.
 If we define $f(x)=x^{Q+1}\exp(-\gamma x)$ for $\gamma>0$, one can easily check that
 $f(x) \le f(\frac{Q+1}{\gamma})=\frac{(Q+1)^{Q+1}\exp\{-(Q+1)\}}{\gamma^{Q+1}}$. Thus, 
\begin{equation}\label{Q}
\sum_{n=l}^{\infty} \frac{H_l}{H_{n+1}} \exp\left(-\frac{D_l}{D_n}\right)
\le C\exp(2\varepsilon l)\sum_{k=0}^{\infty} \alpha^{-k}C'\exp\{4(Q+1)\varepsilon l\}\le C''\exp\{(4Q+6)\varepsilon l\}.
\end{equation}
By combining (\ref{defn}), (\ref{til_l}), (\ref{after_l}) and (\ref{Q}), we obtain
\begin{equation}\label{C}
\log{\frac{1}{H_l}}-1\le
\log{p_t(\omega,\omega)} \le \log{\frac{1}{H_l}}+\log{\big(1+C''\exp\{(4Q+6)\varepsilon l\}\big)}.
\end{equation}
By Theorem 3.1, this implies 
\begin{equation*}
\lim_{t \to 0}\frac{\log{p_t(\omega,\omega)}}{l(\omega,t)}=\beta_{\lambda},\ \ {\rm for\ all}\ \omega \in \Sigma_{*},\ \ \mathbb{P}_{{\rm GW}}\mathchar`-a.s.
\end{equation*}
So the proof will be finished if we prove the following.
\begin{equation}\label{B}
\lim_{t \to 0} -\frac{l(\omega,t)}{\log t}=\frac{1}{\beta_{\lambda}-\log\lambda},\ \ {\rm for\ all}\ \omega \in \Sigma_{*},
\ \ \ \mathbb{P}_{{\rm GW}}\mathchar`-a.s. 
\end{equation}
 We now prove $(\ref{B})$. By the definition of $l(\omega,t)$,
 Theorem \ref{HM} and Proposition \ref{Resistance}, for all $\varepsilon>0$ and $\omega \in \Sigma_{*}$,
 we have the following inequality for $t>0$ sufficiently small $\mathbb{P}_{{\rm GW}}$-a.s.
\begin{align}\label{BB}
\exp \{-(\beta_{\lambda}+\varepsilon)l\}\exp\{(\log\lambda-\varepsilon) l\} &\le H_l R_l=D_l
<t\leq D_{l-1}=H_{l-1}R_{l-1}\nonumber\\
&\le \exp\{-(\beta_{\lambda}-\varepsilon)l\}\exp\{(\log\lambda+\varepsilon)l\}. 
\end{align}
So, we have
\begin{equation}-(\beta_{\lambda}-\log\lambda+2\varepsilon)l<\log t \le -(\beta_{\lambda}-\log\lambda-2\varepsilon)l, \nonumber \end{equation}
and $(\ref{B})$ is proved.
\qed\\

In order to prove Theorem \ref{displacement}, we first show the following proposition corresponding to Theorem \ref{HKE}.

\begin{Proposition}\label{Dis}
For $0<\lambda<m$, the following holds $\mathbb{P}_{{\rm GW}}$-a.s.
\begin{eqnarray*}
\lim_{t \to 0} \frac{\log{E_{\omega}[D(\omega,X_t)^{\gamma}]}}{\log t}=\gamma\wedge1,\ \ \ \ {\rm HARM}_{\mathcal{T}}^{\lambda}\ a.e.
\mathchar`-\omega. 
\end{eqnarray*}
\end{Proposition}
\begin{proof}
Analogously to Theorem \ref{OND}, it is sufficient to prove the claim for any $\omega\in\Sigma_{*}$.\\
{\bf Lower bound.} In order to establish the lower bound, we have to obtain the upper bound of
 $E_{\omega}[D(\omega,X_t)^{\gamma}]$. Define
\begin{equation*} 
I_0 := \{\eta \in \Sigma :D(\omega,\eta) \leq t \},\ I_n  := \{ \eta \in \Sigma :2^{n-1}t < D(\omega,\eta) \le 2^nt \},
\end{equation*}
for $n\geq1$. From the expression of the heat kernel (\ref{expression}), it follows that $p_t(\omega,\omega) \geq p_t(\omega,\eta)$\ for\ all\ $\omega,\eta \in \Sigma$. So by using the on-diagonal upper bound of the heat kernel (\ref{C}),
 for $t>0$ sufficiently small, we have the following for all $\omega \in \Sigma_{*}.$
\begin{equation*}
\int_{I_0} p_t(\omega,\eta)D(\omega,\eta)^{\gamma}d{\rm HARM}(\eta)
 \le p_t(\omega,\omega)\cdot t^{\gamma}\cdot {\rm HARM}(I_0) \le t^{\gamma}(1+C\exp((4Q+6)\varepsilon l)).
\end{equation*}
Note that by Proposition \ref{HK}, we have the off-diagonal upper bound for the heat kernel; 
\begin{equation}p_t(\omega,\eta) \le \frac{t}{D(\omega,\eta){\rm HARM}\left(B_D(\omega,D(\omega,\eta))\right)}. \nonumber \end{equation}
 Thus for $n\geq1$,
\begin{eqnarray*}
\int_{I_n}p_t(\omega,\eta)D(\omega,\eta)^{\gamma}d{\rm HARM}(\eta) &\le& 
t \int_{I_n}\frac{D(\omega,\eta)^{\gamma -1}}{{\rm HARM}(B_D(\omega,D(\omega,\eta)))} d{\rm HARM}(\eta)\\
&\le& t(2^{n}t)^{\gamma-1}\left( \frac{{\rm HARM}(B_D(\omega,2^nt))}{{\rm HARM}(B_D(\omega,2^{n-1}t))}-1\right).
\end{eqnarray*}
By the proof of Theorem \ref{OND},
there\ exists\ $t_0=t_0(\omega,\varepsilon)>0$\ such\ that\ \ $t^{\kappa_{\lambda}+\varepsilon}\le 
{\rm HARM}(B_D(\omega,t))\le t^{\kappa_{\lambda}-\varepsilon}$
for\ all\ $0<t\le t_0$\ and\ for\ all\ $\omega \in \Sigma_{*}$,\ where\ $\kappa_{\lambda}
=\frac{\beta_{\lambda}}{\beta_{\lambda}-\log\lambda}$. 
So we have,
\begin{equation*}
\frac{{\rm HARM}(B_D(\omega,2t))}{{\rm HARM}(B_D(\omega,t))}\le \frac{(2t)^{\kappa_{\lambda}-\varepsilon}}
{t^{\kappa_{\lambda}+\varepsilon}}\le Ct^{-2\varepsilon},\ \ {\rm for\ all}\ 0<t\le t_0,\ {\rm and\ for\ all}\ \omega \in \Sigma_{*},
\end{equation*}
where\ $C$\ is\ a\ constant\ which\ does not\ depend\ on\ $t.$ Define $M:= \max\{n:2^nt \le t_0  \}$\ and\ 
$I:= \bigcup_{n\geq M+1}I_n$. Then we have the following results for $\varepsilon>0$ and $t>0$ sufficiently small.
\begin{eqnarray}\label{I_n}
\sum_{n=1}^M \int_{I_n} p_t(\omega,\eta)D(\omega,\eta)^{\gamma}d{\rm HARM}(\eta)
\le t\sum_{n=1}^M (2^nt)^{\gamma-1}\{C(2^nt)^{-2\varepsilon}-1\} 
\le \begin{cases}
Ct^{\gamma-2\varepsilon}\ 0<\gamma<1\\
Ct^{1-2\varepsilon}\ \ \gamma=1\\
Ct\ \ \ \ \ \ \ \ \gamma>1,
\end{cases}
\end{eqnarray}
and
\begin{eqnarray}\label{I}
&&\int_{I}p_t(\omega,\eta)D(\omega,\eta)^{\gamma}d{\rm HARM}(\eta)
\le t \int_{I} \frac{D(\omega,\eta)^{\gamma -1}}{{\rm HARM}(B_D(\omega,D(\omega,\eta)))}d{\rm HARM}(\eta) \nonumber\\
&\le& t R(o)^{\gamma-1} \cdot \frac{{\rm HARM}(I)}{{\rm HARM}(B_D(\omega,t_0))}
\le t R(o)^{\gamma-1} \cdot \frac{1-{\rm HARM}(B_D(\omega,\frac{t_0}{2}))}{{\rm HARM}(B_D(\omega,t_0))}.
\end{eqnarray}
Since
\begin{equation*}
E_{\omega}[D(\omega,X_t)^{\gamma}] 
=\sum_{n=0}^M \int_{I_n} p_t(\omega,\eta)D(\omega,\eta)^{\gamma}d{\rm HARM}(\eta)
+\int_{I}p_t(\omega,\eta)D(\omega,\eta)^{\gamma}d{\rm HARM}(\eta),
\end{equation*}
by combining (\ref{I_n}) and (\ref{I}), we obtain
\begin{equation*}
\liminf_{t \to 0} \frac{\log{E_{\omega}[D(\omega,X_t)^{\gamma}]}}{\log t}\geq \gamma \wedge 1,\ \ \ {\rm for\ all}\ \omega \in \Sigma_{*}.
\end{equation*}
{\bf Upper bound.} In order to establish the upper bound, we have to obtain a lower bound of $E_{\omega}[D(\omega,X_t)^{\gamma}].$ First, we will prove the following.
\begin{eqnarray}\label{leq1}
\limsup_{t \to 0} \frac{\log{E_{\omega}[D(\omega,X_t)^{\gamma}]}}{\log t} \leq 1,\ \ \ {\rm for\ all}\ \omega\in\Sigma_{*}.
\end{eqnarray}
Define $F_{\omega,\eta}(t):= p_t(\omega,\eta)$\ for $\omega,\eta \in \Sigma$. Then by (\ref{expression}),
 we have 
\begin{eqnarray}\label{F'}
F'_{\omega,\eta}(t)&=&\sum_{n=0}^{N(\omega,\eta)} \dfrac{\bigl(D([\omega]_n)\bigr)^{-1}\exp(-t/D([\omega]_n))-
\bigl(D([\omega]_{n-1})\bigr)^{-1}\exp(-t/D([\omega]_{n-1}))}
{{\rm HARM}(\Sigma([\omega]_n))}, 
\end{eqnarray}
with the convention $(D([\omega]_{-1}))^{-1}=0$.
By (\ref{F'}), we have
\begin{align*}
F_{\omega,\eta}(t)-F'_{\omega,\eta}(t)t
=\sum_{n=0}^{N(\omega,\eta)}\frac{(1+t/D([\omega]_{n-1}))\exp(-t/D([\omega]_{n-1}))-(1+t/D([\omega]_{n}))\exp(-t/D([\omega]_{n}))}
{{\rm HARM}(\Sigma([\omega]_n))}.
\end{align*}
It is easy to show that $g(x):=(1+x)\exp(-x)$ is decreasing for $x>0$.
Since 
\begin{align*}
t/D([\omega]_{n-1})<t/D([\omega]_n), 
\end{align*}
we get
\begin{align*} 
&g(t/D([\omega]_{n-1}))-g(t/D([\omega]_n))\\
&=(1+t/D([\omega]_{n-1}))\exp(-t/D([\omega]_{n-1}))-(1+t/D([\omega]_{n}))\exp(-t/D([\omega]_{n}))\geq0.
\end{align*}
This implies $F_{\omega,\eta}(t)\geq F'_{\omega,\eta}(t)t$  for $t>0$.
Note that $f(\rho)=\rho e^{-\rho t}$ is increasing for $0 \le \rho \le t^{-1}$.
 Thus for $0<t<D(\omega,\eta)$ and $0\le n\le N(\omega,\eta)$, we have
\begin{equation*}
\bigl(D([\omega]_n)\bigr)^{-1}\exp(-t/D([\omega]_n))-
\bigl(D([\omega]_{n-1})\bigr)^{-1}\exp(-t/D([\omega]_{n-1}))\geq0.
\end{equation*}
Therefore, for $0<t<D(\omega,\eta)$ we get
\begin{align*}
F'_{\omega,\eta}(t)\geq \frac{D(o)\exp(-t/D(o))}{{\rm HARM}(\Sigma)}\geq D(o)\exp(-\tilde{D}/D(o)),
\end{align*}
where $\tilde{D}:={\rm diam}(\Sigma,D)$.
Hence, we have
\begin{align*}
p_t(\omega,\eta)\geq D(o)\exp(-\tilde{D}/D(o))t\ \ \ {\rm for}\ 0<t<D(\omega,\eta).
\end{align*}
 This estimate implies that for any $\varepsilon>0$ and $0<t<\varepsilon$, we have
 \begin{align*}
 E_{\omega}[D(\omega,X_t)^{\gamma}]&\geq
 \int_{\{\eta:D(\omega,\eta)>\varepsilon\}}p_t(\omega,\eta)D(\omega,\eta)^{\gamma}d{\rm HARM}(\eta)\\
 &\geq D(o)\exp(-\tilde{D}/D(o))t\cdot \varepsilon^{\gamma}\cdot
 {\rm HARM}(\Sigma\setminus B_D(\omega,\varepsilon)).
 \end{align*}
 This implies (\ref{leq1}).
Now, the proof will be finished once we prove the following. 
\begin{equation}\label{BBB}
\limsup_{t \to 0} \frac{\log{E_{\omega}[D(\omega,X_t)^{\gamma}]}}{\log t} \leq \gamma \ \ {\rm for\ all}\ \omega\in \Sigma_{*}.
\end{equation}
In order to prove $(\ref{BBB})$, we first show the following near diagonal lower bound for the heat kernel.
 For $\varepsilon>0$ sufficiently small and for all $\omega \in \Sigma_{*}$, there exist
 $C_1=C_1(\omega,\varepsilon)>0$, $t_1=t_1(\omega,\varepsilon)\in(0,t_0)$, $\delta '=\delta '(\omega,\varepsilon)>0$
 $\delta ''=\delta ''(\omega,\varepsilon)>0$ with $\delta''>\delta'$ and
 $\delta',\delta'' \rightarrow 0$ as $\varepsilon \rightarrow 0$ such that 
 \begin{equation}\label{BBBB}
 p_t(\omega,\eta)\geq \frac{C_1}{t^{\kappa_{\lambda}-\varepsilon}},
\ {\rm for\ all}\ 0<t\le t_1\ \ {\rm and\ all}\ \eta \in B_D(\omega,t^{1+\delta '})\setminus B_D(\omega,t^{1+\delta ''}). 
\end{equation}
Recall that $\kappa_{\lambda}
=\frac{\beta_{\lambda}}{\beta_{\lambda}-\log\lambda}$.
For $\omega,\eta \in \Sigma$, and $N(\omega,\eta)$ in Definition \ref{def}, we have
\begin{align*}
p_t(\omega,\eta)&=\sum_{n=0}^{N(\omega,\eta)}
 \frac{\exp(-t/D([\omega]_{n-1}))-\exp(-t/D([\omega]_n)}{{\rm HARM}(\Sigma([\omega]_n))}\\
&=p_t(\omega,\omega)-\sum_{n=N(\omega,\eta)+1}^{\infty}
\frac{\exp(-t/D([\omega]_{n-1}))-\exp(-t/D([\omega]_n)}{{\rm HARM}(\Sigma([\omega]_n))}\\
&\geq \frac{C}{{\rm HARM}(B_D(\omega,t))}-\sum_{n=N(\omega,\eta)+1}^{\infty}
 \frac{\exp(-t/D([\omega]_{n-1}))}{{\rm HARM}(\Sigma([\omega]_n))}.
\end{align*}
If $N=N(\omega,\eta)\geq L(\omega,\varepsilon)\vee K(\omega,\varepsilon)$,
 where $L(\omega,\varepsilon)$ and $K(\omega,\varepsilon)$
 are given in the proof of Theorem \ref{OND}, then
\begin{eqnarray*}
&&\sum_{n=N+1}^{\infty} \frac{\exp(-t/D([\omega]_{n-1}))}{{\rm HARM}(\Sigma([\omega]_n))} 
\le C\sum_{n=N+1}^{\infty}\exp\{(\beta_{\lambda}+\varepsilon)\}
\exp[-C't\cdot \exp\{-(\log\lambda+\varepsilon)n\} \cdot\exp\{(\beta_{\lambda}-\varepsilon)n\}]  \\
&\le& C\sum_{n=N+1}^{\infty}\exp\{(\beta_{\lambda}+\varepsilon)n\}
\exp[-C't\exp\{(\beta_{\lambda}-\log\lambda-2\varepsilon) n\}] 
= C\sum_{n=N+1}^{\infty}\alpha^{\delta n}\exp(-C't\alpha^n)\\
&\le& C\int_{\alpha^{N}}^{\infty}x^P\exp(-C'tx)dx,
\end{eqnarray*}
where\ $\alpha=\exp(\beta_{\lambda}-\log\lambda-2\varepsilon)$,
 $\delta=\frac{\beta_{\lambda}+\varepsilon}{\beta_{\lambda}-\log\lambda-2\varepsilon}$
 and $P:=[\delta]+1$. It is easy to check that
 \begin{align*}
\int x^P\exp(-C'tx)dx=\exp(-C'tx)\sum_{i=0}^P(-1)^{P-i}\frac{P!}{i!(-C't)^{P-i+1}}x^i.
\end{align*}
Therefore, we have
\begin{align*}
\sum_{n=N+1}^{\infty} \frac{\exp(-t/D([\omega]_{n-1}))}{{\rm HARM}(\Sigma([\omega]_n))}
\leq C''t^{-P-1}\exp(-C't\alpha^{N})\sum_{i=0}^P|t\alpha^N|^i.
\end{align*}
 When $t^{1+\theta}\le D(\omega,\eta)\le t^{1+\theta'}$ for $0<\theta'<\theta$,
 for $N=N(\omega,\eta)\geq L(\omega,\varepsilon)\vee K(\omega,\varepsilon)$, we have
 $t^{-\theta'}\leq \alpha^Nt\leq t^{-\frac{1+\theta}{\Delta}+1},$
 where $\Delta=\frac{\beta_{\lambda}-\log\lambda+2\varepsilon}{\beta_{\lambda}-\log\lambda-2\varepsilon}$.
 Hence, when $t^{1+\theta}\le D(\omega,\eta)\le t^{1+\theta'}$ for $0<\theta'<\theta$, and $t>0$ sufficiently small,
 \begin{align*}
 \sum_{n=N+1}^{\infty} \frac{\exp(-t/D([\omega]_{n-1}))}{{\rm HARM}(\Sigma([\omega]_n))}
 \leq C''\exp(-C't^{-\theta'})t^{-P-1}\sum_{i=0}^P(t^{-\frac{1+\theta}{\Delta}+1})^i
 \leq C''\exp(-C't^{-\theta'})t^{-\frac{P(1+\theta)}{\Delta}-1}.
 \end{align*}
Since $x^k\exp(-C'x)\leq\frac{k^k\exp(-k)}{(C')^k}$ for any $x>0$ and $k\in\mathbb{N}$,
 we have completed the proof of (\ref{BBBB}).
 Next, define $B:= \{\eta \in \Sigma:t^{1+d} \le D(\omega,\eta)\le t^{1+d'} \}$,
 where $\delta''>d>d'>\delta'>0$. Then according to (\ref{BBBB}), we have
\begin{align*}E_{\omega}[D(\omega,X_t)^{\gamma}]&\geq \int_B p_t(\omega,\eta)D(\omega,\eta)^{\gamma}
d{\rm HARM}(\eta)
\geq \frac{1}{t^{\kappa_{\lambda}-\varepsilon}}\cdot t^{(1+d)\gamma}\cdot
 {\rm HARM}(B)\\
 &\geq t^{\gamma -\kappa_{\lambda}+\gamma d+\varepsilon}
 \{(t^{1+d'})^{\kappa_{\lambda}+\varepsilon}-(t^{1+d})^{\kappa_{\lambda}-\varepsilon}\}. 
\end{align*}
By taking $\varepsilon,\ d$\ and\ $d'$ sufficiently small such that
 $d\kappa_{\lambda}-d\varepsilon>d'\kappa_{\lambda}+d'\varepsilon+2\varepsilon$,
 we get the desired result.
\end{proof}
As a corollary of Proposition \ref{Dis}, we can easily show Theorem \ref{displacement}.\\

{\it Proof of Theorem \ref{displacement}}
\ \ For\ $\omega \in\Sigma_{*},\eta \in \Sigma$,\ $d(\omega,\eta)=\exp(-N(\omega,\eta))$,
 and for\ $N(\omega,\eta)\geq L(\omega,\varepsilon)\vee K(\omega,\varepsilon)$,
 by $(\ref{BB})$ in the proof of Theorem \ref{OND}, we have
\begin{equation*}
\exp\{-(\beta_{\lambda}-\log\lambda+2\varepsilon)N(\omega,\eta)\}
 \le D(\omega,\eta)=D_N \le \exp\{-(\beta_{\lambda}-\log\lambda-2\varepsilon)N(\omega,\eta)\}.
\end{equation*}
Hence for\ $d(\omega,\eta)\le \exp\{-L(\omega,\eta)\vee K(\omega,\eta)\}$, we have
\begin{equation}
d(\omega,\eta)^{\beta_{\lambda}-\log\lambda+2\varepsilon} \le D(\omega,\eta) \le d(\omega,\eta)^{\beta_{\lambda}-\log\lambda-2\varepsilon}.
\end{equation}
This implies that there\ exists\ $C_2=C_2(\omega,\varepsilon)>0$\ \ such\ that
\begin{equation*}
C_2^{-1}D(\omega,\eta)^{\frac{1}{\beta_{\lambda}-\log\lambda-2\varepsilon}}\le d(\omega,\eta) \le C_2D(\omega,\eta)^{\frac{1}{\beta_{\lambda}-\log\lambda+2\varepsilon}}. 
\end{equation*}
This estimate and Proposition \ref{Dis} yield the claim.\qed


\begin{thebibliography}{99}

\bibitem{AN} Athreya, K.B., Ney, P.:
{\it Branching Processes},
 Springer, Heidelberg, (1972).

\bibitem{B} Baxter, M.:
{\it Markov processes on the boundary of the binary tree,}
 in Lecture Notes in Math, vol. 1526, Springer-Verlag, 1992, 210-244.

\bibitem{C} Cartier, C.:
{\it Fonctions harmoniques sur un arbre},
 in Symp. Math., vol. 9, Academic Press, (1972), 203-270.


\bibitem{D} Doob, J.L.:
{\it Boundary properties for functions with finite Dirichlet integrals},
Ann. Inst. Fourier (Grenoble), 12 (1962), 573-621.


\bibitem{K} Kigami, J.:
{\it Dirichlet forms and associated heat kernels on the Cantor set induced by random walks on trees},
 Adv. in Math. 225, (2010), 2674-2730.

\bibitem{Lin} Lin, S.:
{\it Harmonic measure for biased random walk in a supercritical Galton-Watson tree},
Arxiv:1707.01811, to appear in Bernoulli Journal.

\bibitem{L} Lyons, R.:
{\it Random walks and percolation on trees},
 Ann. Prob. 20, (1990), 931-958.

\bibitem{L1} Lyons, R.:
{\it Equivalence of boundary measures on covering trees of finite graphs},
 Ergodic Theory Dynamical Systems, 14, (1994), 575-597.

\bibitem{LP} Lyons, R., Peres, Y.:
{\it Probability on Trees and Networks},
Cambridge Series in Statistical and Probabilistic Mathematics, 42. Cambridge University Press, New York, (2016). 

\bibitem{LPP1} Lyons, R., Pemantle, R., Peres, Y.:
{\it Ergodic theory on Galton-Watson trees: speed of random walk and dimension of harmonic measure},
 Ergodic Theory Dynamical Systems, 15, (1995), 593-619.

\bibitem{LPP2} Lyons, R. Pemantle, R., Peres, Y.:
{\it Biased random walks on Galton-Watson trees},
 Probab. Theory Relat. Fields, 106, (1996), 249-264.

\bibitem{LPP3} Lyons, R. Pemantle, R., Peres, Y.:
{\it Unsolved problems concerning random walks on trees},
 Classical and modern branching processes (Minneapolis, MN, 1994),  223-237, IMA Vol. Math. Appl., 84, Springer, New York, 1997.
 
\bibitem{N}
 Na\"im, L.:
 {\it Sur le role de la frontiere de R. S. Martin dans la theorie du potential},
 Annales Inst. Fourier. {\bf 7} (1957), 183-281.
 
 \bibitem{Rou} Rousselin, P.:
 {\it Invariant Measures, Hausdorff Dimension and Dimension Drop of some Harmonic Measures on Galton-Watson Trees},
 Electron. J. Probab. 23 (2018), Paper No. 46, 31 pp.
 
 \bibitem{S}
 Silverstein, M.L.
{\it Classification of stable symmetric Markov chains},
 Indiana J. Math. {\bf 24} (1974), 29--77. 
 
\bibitem{W} Woess, W.:
{\it Denumerable Markov Chains},
 European Math. Soc., (2009)

\end{thebibliography}
\end{document}